\documentclass[reqno]{amsart}
\usepackage{hyperref}

\def\N{{\mathbb{N}}}

\def\R{{\mathbb{R}}}

\begin{document}
\title[Pontryagin]{Discrete time Pontryagin principles in Banach spaces}

\author[Bachir and Blot]
{Mohammed Bachir and Jo${\rm \ddot e}$l Blot} 

\address{Mohammmed Bachir: Laboratoire SAMM EA4543,\newline
Universit\'{e} Paris 1 Panth\'{e}on-Sorbonne, centre P.M.F.,\newline
90 rue de Tolbiac, 75634 Paris cedex 13,
France.}
\email{Mohammed.Bachir@univ-paris1.fr}

\address{Jo\"{e}l Blot: Laboratoire SAMM EA4543,\newline
Universit\'{e} Paris 1 Panth\'{e}on-Sorbonne, centre P.M.F.,\newline
90 rue de Tolbiac, 75634 Paris cedex 13,
France.}
\email{blot@univ-paris1.fr}
\date{May 15th 2017}

\numberwithin{equation}{section}
\newtheorem{theorem}{Theorem}[section]
\newtheorem{lemma}[theorem]{Lemma}
\newtheorem{example}[theorem]{Example}
\newtheorem{remark}[theorem]{Remark}
\newtheorem{definition}[theorem]{Definition}
\newtheorem{proposition}[theorem]{Proposition}
\newtheorem{corollary}[theorem]{Corollary}
\begin{abstract}
The aim of this paper is to establish Pontryagin's principles in a dicrete-time infinite-horizon setting when the state variables and the control variables belong to infinite dimensional Banach spaces. In comparison with previous results on this question, we delete conditions of finiteness of codimension of subspaces. To realize this aim, the main idea is the introduction of new recursive assumptions and useful consequences of the Baire category theorem and of the Banach isomorphism theorem.
\end{abstract}
\maketitle
M.S.C. 2000: 49J21, 65K05, 39A99.
\section{Introduction}
The considered infinite-horizon Optimal Control problems are governed by the following discrete-time controlled dynamical system.
\begin{equation}\label{eq11}
x_{t+1} = f_t(x_t,u_t), \hskip3mm t \in \N
\end{equation}
where $x_t \in X_t \subset X$, $u_t \in U_t \subset U$ and $f_t : X_t \times U_t \rightarrow X_{t+1}$. Here $X$ and $U$ are real Banach spaces; $X_t$ is a nonempty open subset of $X$ and $U_t$ is a nonempty subset of $U$. As usual, the $x_t$ are called the state variables and the $u_t$ are called the control variables.
\vskip1mm
From an initial state $\sigma \in X_0$, we denote by $Adm(\sigma)$ the set of the processes $((x_t)_{t \in \N}, (u_t)_{t \in \N}) \in (\prod_{t \in \N} X_t) \times (\prod_{t \in \N} U_t)$ which satisfy (\ref{eq11}) for all $t \in \N$. The elements of $Adm(\sigma)$ are called the admissible processes.
\vskip1mm
For all $t \in \N$, we consider the function $\phi_t : X_t \times U_t \rightarrow \R$ to define the criteria.We denote by $Dom(J)$ the set of the $((x_t)_{t \in \N}, (u_t)_{t \in \N}) \in (\prod_{t \in \N} X_t) \times (\prod_{t \in \N} U_t)$ such that the series $\sum_{t=0}^{+ \infty} \phi_t(x_t,u_t)$ is convergent in $\R$. We define the nonlinear functional $J : Dom(J) \rightarrow \R$ by setting
\begin{equation}\label{12}
J((x_t)_{t \in \N}, (u_t)_{t \in \N}) := \sum_{t=0}^{+ \infty} \phi_t(x_t,u_t).
\end{equation}
Now we can give the list of the considered problems of Optimal Control.
\vskip2mm
\noindent
$({\bf P}_1(\sigma))$: Find $((\hat{x}_t)_{t \in \N}, (\hat{u}_t)_{t \in \N}) \in Dom(J) \cap Adm(\sigma)$ such that $J((\hat{x}_t)_{t \in \N}, (\hat{u}_t)_{t \in \N}) \geq J((x_t)_{t \in \N}, (u_t)_{t \in \N})$ for all 
$((x_t)_{t \in \N}, (u_t)_{t \in \N}) \in Dom(J) \cap Adm(\sigma)$.
\vskip2mm
\noindent
$({\bf P}_2(\sigma))$: Find $((\hat{x}_t)_{t \in \N}, (\hat{u}_t)_{t \in \N}) \in Adm(\sigma)$ such that \\
 $\limsup_{h \rightarrow + \infty} \sum_{t=0}^h (\phi_t(\hat{x}_t, \hat{u}_t) - \phi(x_t,u_t)) \geq 0$ for all 
$((x_t)_{t \in \N}, (u_t)_{t \in \N}) \in Adm(\sigma)$.
\vskip2mm
\noindent
$({\bf P}_3(\sigma))$: Find $((\hat{x}_t)_{t \in \N}, (\hat{u}_t)_{t \in \N}) \in Amd(\sigma)$ such that \\
 $\liminf_{h \rightarrow + \infty} \sum_{t=0}^h (\phi_t(\hat{x}_t, \hat{u}_t) - \phi(x_t,u_t)) \geq 0$ for all 
$((x_t)_{t \in \N}, (u_t)_{t \in \N}) \in  Adm(\sigma)$.
\vskip1mm
These problems are classical in mathematical macroeconomic theory; cf. \cite{Mi}, \cite{BH}, \cite{Zas}, \cite{SLP} and references therein, and also in sustainable development theory, \cite{Cl}.
\vskip1mm
We study the necessary optimality conditions for these problems in the form of Pontryagin principles. Among the different ways to treat such a question, we choose the method of the reduction to the finite horizon. This method comes from \cite{BCh} in the discrete-time framework. Notice that this viewpoint was previously used by Halkin (\cite{CHL}, Theorem 2.3, p. 20) in the continuous-time framework.
\vskip1mm
There exist several works on this method when $X$ and $U$ are finite dimensional, cf. \cite{BH}. In the present paper we treat the case where $X$ and $U$ are infinite dimensional Banach spaces. With respect to two previous papers on this question, \cite{BB1} and \cite{BB2}, the main novelty is to avoid the use of assumptions of finiteness of the codimension of certain vector subspaces. To realize this we introduce new recursive assumptions on the partial differentials of the $f_t$ of (\ref{eq11}). We speak of recursive assumptions since they contain two successive dates $t-1$ and $t$.
\vskip1mm
To make more easy the reading of the paper we describe the schedule of the proof of the main theorem (Theorem \ref{th21} below) .\\
\underline{First step}: the method of the reduction to finite horizon associates to the considered problems in infinite horizon the same sequence of finite-horizon problems which is indexed by $h \in \N$, $h \geq 2$.\\
\underline{Second step}: the providing of conditions to ensure that we can use Multiplier Rules (in Banach spaces) on the finite-horizon problems. Hence we obtain, for each $h \in \N$, $h \geq 2$, a nonzero list $(\lambda_0^h, p_1^h, ..., p_{h+1}^h) \in \R \times (X^*)^{h+1}$ where $\lambda_0^h$ is a multiplier associated to the criterion and $(p_1^h, ..., p_{h+1}^h)$ are multipliers associated to the (truncated) dynamical system which is transformed into a list of constraints.\\
\underline{Third step}: the building of an increasing function $\varphi : \N \rightarrow \N$ such that the subsequences $(\lambda_0^{\varphi(h)})_{h}$ and $(p_{t+1}^{\varphi(h)})_h$ respectively converge to $\lambda_0$ and $p_{t+1}$ for each $t \in \N_*$, with $(\lambda_0, (p_{t+1})_t)$ nonzero. The Banach-Alaoglu theorem permits us to obtain weak-star convergent subsequences of $(\lambda_0^h)_h$ and  $(p_{t+1}^h)_h$ for each $t \in \N$, and a diagonal process of Cantor permits us to obtain the same function $\varphi$ for all $t \in \N$. The main difficulty is to avoid that $(\lambda_0, (p_{t+1})_t)$ is equal to zero. Such a difficulty is due to the infinite dimension where the weak-star closure of a sphere centered at zero contains zero. To overcome this difficulty, using the Baire category theorem, we establish that a weak-star convergence implies a norm convergence on a well chosen Banach subspace of the dual space of the state space.
\vskip2mm
Now we describe the contents of the paper. In Section 2 we present our assumptions and we give the statement of the main theorem on the Pontryagin principle. In Section 3 we recall a characterization of the closedness of the image of a linear continuous operator, a consequence of the Baire category theorem on the weak-star convergence, and we provide a diagonal process of Cantor for the weak-star convergence. In Section 4 we describe the reduction to the finite horizon and we establish consequence of our recursive assumptions on the surjectivity and on the closedness of the range of the differentials of the constraints in the finite-horizon problems. In Section 5 we give the complete proof of our main theorem.
\section{The main result}
First we present a list of hypotheses.
\vskip2mm
\noindent
{\bf (H1)}: $X$ and $U$ are separable Banach spaces.
\vskip2mm
\noindent
{\bf (H2)}: For all $t \in \N$, $X_t $ is a nonempty open subset of $X$ and $U_t$ is a nonempty convex subset of $U$.
\vskip2mm
\noindent
When $((\hat{x}_t)_{t \in \N}, (\hat{u}_t)_{t \in \N})$ is a given admissible process of one of the problems ($({\bf P}_i(\sigma))$), $i \in \{1,2,3 \}$, we consider the following conditions.
\vskip2mm
\noindent
{\bf (H3)}: For all $t \in \N$, $\phi_t$ is Fr\'echet differentiable at $(\hat{x}_t, \hat{u}_t)$ and $f_t$ is continuously Fr\'echet differentiable at $(\hat{x}_t, \hat{u}_t)$.
\vskip2mm
\noindent
{\bf (H4)}:  For all $t \in \N$, $t \geq 2$, \\
$D_1f_t(\hat{x}_t, \hat{u}_t) \circ D_2 f_{t-1}(\hat{x}_{t-1}, \hat{u}_{t-1})(U) + D_2f_t(\hat{x}_t, \hat{u}_t)(T_{U_t}(\hat{u}_t)) = X$.
\vskip2mm
\noindent
{\bf (H5)}: $D_1f_1(\hat{x}_1, \hat{u}_1) \circ D_2 f_0( \hat{x}_0, \hat{u}_0)( T_{U_0}(\hat{u}_0)) + D_2f_1(\hat{x}_1, \hat{u}_1)(T_{U_1}(\hat{u}_1)) = X$.
\vskip2mm
\noindent
{\bf (H6)}: $ri( T_{U_0}(\hat{u}_0)) \neq \emptyset$ and $ri(T_{U_1}(\hat{u}_1)) \neq \emptyset$.
\vskip2mm
In (H3), since $U_t$ is not necessarily a neighborhood of $\hat{u}_t$, the meaning of this condition is that there exists an open neighborhood $V_t$ of $(\hat{x}_t, \hat{u}_t)$ in $X \times U$ and a Fr\'echet differentiable function (respectively continuously Fr\'echet differentiable mapping) $\tilde{\phi}_t : V_t \rightarrow \R$ (respectively $\tilde{f}_t : V_t \rightarrow X$) such that ${\tilde{\phi}_t}$ and ${\phi_t}$ (respectively ${\tilde{f}_t}$ and ${f_t}$) coincide on $V_t \cap (X_t \times U_t)$. Moreover $D_1$ and $D_2$ denotes the partial Fr\'echet differentials with respect to the first (vector) variable and with respect to the second (vector) variable respectively. About (H4), (H5) and (H6), when $A$ is a convex subset of $U$, $\hat{u} \in A$, the set $T_{A}(\hat{u})$  is the closure of $\R_+(A - \hat{u})$; it is  called the tangent cone of $A$ at $\hat{u}$ as it is usually defined in Convex Analysis, \cite{AE} p. 166. About (H6), if ${\rm aff}(T_{U_t}(\hat{u}_t))$ denotes the affine hull of $T_{U_t}(\hat{u}_t)$, $ri(T_{U_t}(\hat{u}_t))$ denotes the (relative) interior of $T_{U_t}(\hat{u}_t)$ in  ${\rm aff}(T_{U_t}(\hat{u}_t))$. Such definition of the relative interior of a convex is given in \cite{Zal}, p. 14-15, where it is denoted by $rint$.
\vskip2mm
Now we state the main result of the paper.
\vskip2mm
\begin{theorem}\label{th21} Let $((\hat{x}_t)_{t \in \N}, (\hat{u}_t)_{t \in \N})$ be an optimal process for one of the problems $({\bf P}_i(\sigma))$, $i \in \{1,2,3 \}$. Under (H1-H6), there exist $\lambda_0 \in \R$ and $(p_{t+1})_{t \in \N} \in (X^*)^{\N}$ which satisfy the following conditions.
\begin{enumerate}
\item $(\lambda_0, p_1,p_2) \neq (0,0,0)$.
\item $\lambda_0 \geq 0$.
\item $p_t = p_{t+1} \circ D_1f_t(\hat{x}_t, \hat{u}_t) + \lambda_0 D_1 \phi_t(\hat{x}_t, \hat{u}_t)$, for all $t \in \N$, $t \geq 1$.
\item $\langle \lambda_0 D_2\phi_t(\hat{x}_t, \hat{u}_t) + p_{t+1} \circ D_2 f_t(\hat{x}_t, \hat{u}_t), u_t - \hat{u}_t \rangle \leq 0$, for all $u_t \in U_t$, for all $t \in \N$.
\end{enumerate}
\end{theorem}
\vskip2mm
In comparison with Theorem 2.2 in \cite{BB2}, in this theorem we have deleted the condition of finiteness of codimension which are present in assumptions (A5) and (A6) in \cite{BB2}.  It is why this theorem is an improvment of the result of \cite{BB2}.
\section{Functional analytic results}
In this section, first we recall an characterization of the closedness of the image of a linear continuous operator. Secondly
we state a result which is a consequence of the Baire category theorem. After we give a version of the diagonal process of Cantor for the weak-star convergence.
\begin{proposition}\label{prop31}
Let $E$ and $F$ be Banach spaces, and $L \in \mathfrak{L}(E,F)$ (the space of the linear continuous mappings). The two following assertions are equivalent.
\begin{enumerate}
\item[(i)] $Im L$ is closed in $F$.
\item[(ii)] There exists $c \in (0, + \infty)$ s.t. for all $y \in Im L$, there exists $x_y \in E$ verifying $L x_y = y$ and $\Vert y \Vert \geq c \Vert x_y \Vert$.
\end{enumerate}
\end{proposition}
This result is proven in \cite{BB1} (Lemma 3.4) and in \cite{BCi} (Lemma 2.1).
\vskip2mm
\begin{proposition}\label{prop32}
Let $Y$ be a real Banach space; $Y^*$ is its topological dual space. Let $(\pi_h)_{h \in \N} \in (Y^*)^{\N}$ and $(\rho_h)_{h \in \N} \in (\R_+)^{\N}$. Let $K$ be a nonempty closed convex subset of $Y$ such that $ri(K) 
\neq \emptyset$. Let $a \in K$ and we set  $S := \overline{\rm aff}(K) - a$ which is a Banach subspace. We assume that the following conditions are fulfilled.
\begin{enumerate} 
\item $\rho_h \rightarrow 0$ when $h \rightarrow + \infty$.
\item $\pi_h \overset{w^*}{\rightarrow} 0$ (weak-star convergence) when $h \rightarrow + \infty$.
\item For all $y \in K$, there exists $c_y \in \R$ such that $\pi_h(y) \leq c_y \rho_h$ for all $h \in \N$.
\end{enumerate}
Then we have $\Vert {\pi_h}_{\mid_S} \Vert_{S^*} \rightarrow 0$ when $h \rightarrow + \infty$.
\end{proposition}

This result is established in \cite{BB2} (Proposition 3.5) where several consequences and generalizations are provided.
In the following result, when $t \in \N$, we set $[t, + \infty)_{\N} := [t, + \infty) \cap \N$ and $\N_* := [1, + \infty)_{\N}$.
\vskip2mm
\begin{proposition}\label{prop33} Let $Y$ be a real Banach space; $Y^*$ is its topological dual space. For every $(t,h) \in \N \times \N_*$ such that $t \leq h$ we consider an element $\pi_{t+1}^h \in Y^*$.
We assume that, for every $t \in \N$, the sequence $(\pi_{t+1}^h)_{h \in [t, + \infty)_{\N}}$ is bounded in $Y^*$. Then there exists an increasing function $\beta : \N_* \rightarrow \N_*$ such that, for all $t \in  \N$, there exists $\overline{\pi}_{t+1} \in Y^*$ verifying $\pi_{t+1}^{\beta (h)}  \overset{w^*}{\rightarrow} \overline{\pi}_{t+1}$ when $h \rightarrow + \infty$.
\end{proposition}
\begin{proof}
Using the Banach-Alaoglu theorem, since $(\pi_1^h)_{h \in [0, + \infty)_{\N}}$ is bounded in $Y^*$, there exists an increasing function $\alpha_1 : [0, + \infty)_{\N} \rightarrow[0, + \infty)_{\N}$ and $\overline{\pi}_1 \in Y^*$ such that $
\pi_1^{\alpha_1(h)} \overset{w^*}{\rightarrow} \overline{\pi}_1$ when $h \rightarrow + \infty$. Using the same argument, since $(\pi_2^{\alpha_1(h)})_{h \in [1, + \infty)_{\N}}$ is bounded, there exists an increasing function $\alpha_2 : [1, + \infty)_{\N} \rightarrow [1, + \infty)_{\N}$ and $\overline{\pi}_2 \in Y^*$ such that $\pi_2^{\alpha_1 \circ \alpha_2 (h)} \overset{w^*}{\rightarrow} \overline{\pi}_2$ when $h \rightarrow + \infty$. Iterating the reasoning, for every $t \in  \N_*$, there exist an increasing function $\alpha_t : [t, + \infty)_{\N} \rightarrow [t, + \infty)_{\N}$ and $\overline{\pi}_{t+1} \in Y^*$ such that $\pi_{t+1}^{\alpha_1 \circ ... \circ \alpha_t(h) } \overset{w^*}{\rightarrow} \overline{\pi}_{t+1}$ when $h \rightarrow + \infty$. We define the function $\beta : [0, + \infty)_{\N} \rightarrow  [0, + \infty)_{\N}$ by setting $\beta(h) := \alpha_1 \circ ... \circ \alpha_h(h)$. we arbitrarily fix $t \in \N_*$  and we define the function $\delta_t : [t, + \infty)_{\N} \rightarrow [t, + \infty)_{\N}$ by setting $\delta_t(t) := t$ and $\delta_t(h) := \alpha_{t+1} \circ ... \circ \alpha_h(h)$ when $h > t$.
When $h = t$, we have $\delta_t(t+1) = \alpha_{t+1}(t+1) \geq t+1 > t = \delta_t(t)$. When $h \in [t+1, + \infty)_{\N}$, we have $\alpha_{t+1}(h+1) \geq h+1 > h$ which implies 
$$\delta_t(h+1) = (\alpha_{t+1} \circ ... \circ \alpha_h)( \alpha_{t+1}(h+1)) >  (\alpha_{t+1} \circ ... \circ \alpha_h)(h) = \delta_t(h)$$
since $(\alpha_{t+1} \circ ... \circ \alpha_h)$ is increasing. Hence we have proven that $\delta_t$ is increasing. Since $\beta_{\mid_{[t, + \infty)_{\N}}} = (\alpha_1 \circ ... \circ \alpha_t) \circ \delta_t$, we can say that $(\pi_{t+1}^{\beta(h)})_{h \in [t, + \infty)_{\N}}$ is a subsequence of  $(\pi_{t+1}^{\alpha_1 \circ ... \circ \alpha_t(h)})_{h \in [t, + \infty)_{\N}}$,  we obtain 
$\pi_{t+1}^{\beta (h)}  \overset{w^*}{\rightarrow} \overline{\pi}_{t+1}$ when $h \rightarrow + \infty$.
\end{proof}
\section{reduction to the finite horizon}
When  $((\hat{x}_t)_{t \in \N}, (\hat{u}_t)_{t \in \N})$ is an optimal process for one of the problems (${\bf P}_i(\sigma)$), $i \in \{1,2,3 \}$. The method of the rediction to finite horizon consists on considering of the 
 sequence of the following finite-horizon problems.
\[ ({\bf F}_h(\sigma))
\left\{
\begin{array}{rl}
{\rm Maximize} & J_h(x_1,...,x_h,u_0,...,u_h) := \sum_{t=0}^h \phi_t(x_t,u_t)\\
{\rm when}& (x_t)_{1 \leq t \leq h} \in \prod_{t=1}^h X_t,  (u_t)_{0 \leq t \leq h} \in \prod_{t=0}^h U_t \\
\null & \forall t \in \{0,...,h \}, x_{t+1} = f_t(x_t,u_t)\\
\null & x_0 = \sigma, x_{h+1} = \hat{x}_{t+1}
\end{array}
\right.
\]
The proof of the following lemma is given in \cite{BCh}.
\begin{lemma}\label{lem41} 
When $((\hat{x}_t)_{t \in \N}, (\hat{u}_t)_{t \in \N})$ is an optimal process for one of the problems (${\bf P}_i(\sigma)$), $i \in \{1,2,3 \}$, then, for all $h \in \N_*$, $(\hat{x}_1, ...,  \hat{x}_h, \hat{u}_0, ..., \hat{u}_h)$ is an optimal solution of $({\bf F}_h(\sigma))$.
\end{lemma}
Notice that this result does not need any special assumption. Now we introduce notation to work on these problems. We write ${\bf x}^h := (x_1,...,x_h) \in \prod_{t=1}^h X_t$, ${\bf u}^h := (u_0,...,u_h) \in \prod_{t=0}^h U_t$. For all $h \in \N_*$ and for all $t \in \N$, we introduce the mapping $g_t^h : (\prod_{t=1}^h X_t) \times (\prod_{t=0}^h U_t) \rightarrow X_{t+1}$ by setting
\begin{equation}\label{eq41}
g_t^h({\bf x}^h,{\bf u}^h) := 
\left\{
\begin{array}{lcl}
-x_1 + f_0(\sigma, u_0) & {\rm if} & t=0\\
-x_{t+1} + f_t(x_t,u_t) & {\rm if} & t \in \{ 1,...,h-1 \} \\
- \hat{x}_{h+1} + f_h(x_h,u_h).
\end{array}
\right.
\end{equation}
We introduce the mapping $g^h : (\prod_{t=1}^h X_t) \times (\prod_{t=0}^h U_t) \rightarrow X^{h+1}$ defined by
\begin{equation}\label{eq42}
g^h({\bf x}^h,{\bf u}^h) := (g_0^h({\bf x}^h,{\bf u}^h), ..., g_h^h({\bf x}^h,{\bf u}^h)).
\end{equation}
Under (H3), $g^h$ is of class $C^1$. We introduce the following conditions on the differentials of the $f_t$. 
\begin{equation}\label{eq43}
\forall t \in \N, \;\; ImDf_t(\hat{x}_t, \hat{u}_t) \;\; {\rm is} \;\; {\rm closed} \;\; {\rm in} \;\; X.
\end{equation}
\begin{equation}\label{eq44}
\forall t \in \N_*,\\
Im(D_1f_t(\hat{x}_t, \hat{u}_t) \circ D_2 f_{t-1}(\hat{x}_{t-1}, \hat{u}_{t-1})) + Im D_2f_t(\hat{x}_t, \hat{u}_t) = Im Df_t(\hat{x}_t, \hat{u}_t).
\end{equation}
\begin{equation}\label{eq45}
\forall t \in \N, t \geq 2, \;\; ImDf_t(\hat{x}_t, \hat{u}_t) = X.
\end{equation}
\begin{equation}\label{eq46}
Im (D_1 f_1(\hat{x}_1, \hat{u}_1) \circ D_2 f_0(\sigma, \hat{u}_0)) + Im D_2f_1(\hat{x}_1, \hat{u}_1) =X.
\end{equation}
\begin{lemma}\label{lem42} We assume that (H3) is fulfilled.
\begin{enumerate}
\item[(i)] Under (\ref{eq43}) and (\ref{eq44}), $ImDg^h({\bf x}^h,{\bf u}^h)$ is closed.
\item[(ii)] Under (\ref{eq45}) and (\ref{eq46}), $Dg^h({\bf x}^h,{\bf u}^h)$ is surjective.
\end{enumerate}
\end{lemma}
\begin{proof} {\bf (i)}
To abridge the writing we set $D \hat{f}_t := D f_t(\hat{x}_t, \hat{u}_t)$ and $D_i \hat{f}_t := D_i f_t(\hat{x}_t, \hat{u}_t)$ when $i \in \{1,2 \}$. The condition (H3) implies that $g^h$ is Fr\'echet differentiable at $({\bf x}^h, {\bf u}^h)$.
\vskip2mm
\noindent
We arbitrarily fix ${\bf z}^h = (z_0,...,z_h) \in Im Dg^h({\bf x}^h, {\bf u}^h)$. Therefore there exists ${\bf y}^{h,0} = (y_1^{0}, ..., y_h^{0}) \in X^h$ and ${\bf v}^{h,0} = (v_0^{0},...,v_h^{0}) \in U^{h+1}$ such that\\
  ${\bf z}^h = Dg^T({\bf x}^h, {\bf u}^h)({\bf y}^{h,0}, {\bf v}^{h,0})$ which is equivalent to the set of the three following equations
\begin{equation}\label{eq47}
- y_1^0 + D_2f_0(\sigma, \hat{u}_0) v_0^0 = z_0
\end{equation}
\begin{equation}\label{eq48}
\forall t \in \{1,...,h-1 \}, \;\;
- y_{t+1}^0 + D_1 \hat{f}_t y_t^0 + D_2 \hat{f}_t v_t^0 = z_t
\end{equation}
\begin{equation}\label{eq49}
D_1 \hat{f}_h y_h^0 + D_2 \hat{f}_h v_h^0 = z_h.
\end{equation}
We introduce the linear continuous operator $L_0 \in {\mathfrak L}(X \times U, X)$ by setting 
\begin{equation}\label{eq410} L_0(y_1,v_0) := - y_1 + D_2 \hat{f}_0 v_0.
\end{equation}
Notice that $L_0$ is surjective since $L_0(X \times \{ 0 \}) = X$; therefore $ImL_0$ is closed in $X$. From (\ref{eq47}) we have $z_0 \in Im L_0$. Using Proposition \ref{prop31} on $L_0$ we know that 
\[
\left\{
\begin{array}{l}
\exists a_0 \in (0, + \infty), \forall z_0 \in X, \exists y_1^* \in X, \exists v_0^* \in U \:\;\; {\rm s.t.} \; \; L_0(y_1^*, v_0^*) = z_0 \\
{\rm and} \;\; \max \{ \Vert y_1^* \Vert, \Vert v_0^* \Vert \} \leq a_0 \cdot \Vert z_0 \Vert
\end{array}
\right.
\]
i.e. we have proven
\begin{equation}\label{eq411}
\left.
\begin{array}{l}
\exists a_0 \in (0, + \infty), \exists y_1^* \in X, \exists v_0^* \in U \:\;\; {\rm s.t.} \\
- y_1^* + D_2 \hat{f}_0 v_0^* = z_0 \;\; {\rm and} \;\; \max \{ \Vert y_1^* \Vert, \Vert v_0^* \Vert \} \leq a_0 \cdot \Vert z_0 \Vert
\end{array}
\right\}
\end{equation}
It is important to notice that $a_0$ does not depend on $z_0$.
\vskip2mm
\noindent
We introduce the linear continuous operator $L_1 \in {\mathfrak L}(X \times U, X)$ by setting 
$$L_1(y_2, v_1) := - y_2 + D_2 \hat{f}_1 v_1.$$
Since $L_1(X \times \{ 0 \}) = X$, $L_1$ is surjective and hence $z_1 - D_1 \hat{f}_1 y_1^* \in Im L_1$. Using Proposition \ref{prop31} on $L_1$, we obtain
\[
\left\{
\begin{array}{l}
\exists b_1 \in (0, + \infty), \exists y_2^* \in X, \exists v_1^* \in U \;\; {\rm s.t.}\\
L_1(y_2^*, v_1^*) = z_1 - D_1 \hat{f}_1 y_1^* \;\; {\rm and} \\
\max \{ \Vert y_2^* \Vert, \Vert v_1^* \Vert \} \leq b_1 \cdot \Vert z_1 - D_1 \hat{f}_1 y_1^* \Vert.
\end{array}
\right.
\]
Using (\ref{eq411}) we deduce from the last inequality
\[
\begin{array}{l}
\max \{ \Vert y_2^* \Vert, \Vert v_1^* \Vert \} \leq b_1 \cdot ( \Vert z_1 \Vert + \Vert D_1 \hat{f}_1\Vert \cdot \Vert y_1^* \Vert) 
\leq b_1 \cdot ( \Vert z_1 \Vert + \Vert D_1 \hat{f}_1\Vert \cdot a_0 \cdot \Vert z_0 \Vert)\\
\leq b_1 \cdot (1 + a_0 \cdot \Vert D_1 \hat{f}_1\Vert) \cdot \max \{ \Vert z_0 \Vert, \Vert z_1 \Vert \}.
\end{array}
\]
We set $a_1 := \max \{ a_0,  b_1 \cdot (1 + a_0 \cdot \Vert D_1 \hat{f}_1\Vert) \}$, and then we have proven the following assertion.
\begin{equation}\label{eq412}
\left.
\begin{array}{l}
\exists a_1 \in (0, + \infty), \exists (y_1^*, y_2^*, v_0^*, v_1^*) \in X^2 \times U^2 \;\; {\rm s.t.} \\
- y_1^* + D_2 \hat{f}_0 v_0^* = z_0, \;\;
- y_2^* + D_1 \hat{f}_1 y_1^* + D_2 \hat{f}_1 v_1^* = z_1,\\
\max \{ \vert y_1^* \Vert, \Vert y_2^* \Vert, \Vert v_0^* \Vert, \Vert v_1^* \Vert \} \leq a_1 \cdot \max \{ \Vert z_0 \Vert, \Vert z_1 \Vert \}. 
\end{array}\right\}
\end{equation}
It is important to notice that $a_1$ does not depend on $z_0$, $z_1$. We iterate the reasoning until $h-2$ and we obtain
\begin{equation}\label{eq413}
\left.
\begin{array}{l}
\exists a_{h-2} \in (0, + \infty), \exists (y^*_t)_{1 \leq t \leq h-1} \in X^{h-1}, \exists (v^*_t)_{0 \leq t \leq h-2} \in U^{h-1} \; {\rm s.t.} \\
-y^*_1 + D_2 \hat{f}_0 v^*_0 = z_0, \;\;
\forall t \in \{1,...,h-2 \}, -t^*_{t+1} + D_1 \hat{f}_t y^*_t + D_2 \hat{f}_t v^*_t = z_t\\
\max \{ \max_{1 \leq t \leq h-1}\Vert y^*_t \Vert, \max_{0 \leq t \leq h-2} \Vert v^*_t \Vert \} \leq a_{h-2} \max_{0 \leq t \leq h-2} \Vert z_t \Vert.
\end{array}
\right\}
\end{equation} 
From (\ref{eq49}) we know that $z_h \in Im D\hat{f}_h$. Moreover we have 
$$D_1 \hat{f}_h z_{h-1} \subset Im D\hat{f}_h \;\; {\rm and} \; D_1 \hat{f}_h \circ D_1 \hat{f}_{h-1} y^*_{h-1} \in Im D_1 \hat{f}_h \subset Im D \hat{f}_h$$
and therefore we have
\begin{equation}\label{eq414}
z_h + D_1 \hat{f}_{h-1}z_{h-1} - D_1 \hat{f}_h \circ D_1 \hat{f}_{h-1} y^*_{h-1} \in Im D \hat{f}_h.
\end{equation}
Introduce the linear continuous operator $\Lambda \in {\mathfrak L}(U \times U, X)$ by setting 
\begin{equation}\label{eq415}
\Lambda (v,w) := D_1 \hat{f}_h \circ D_2 \hat{f}_{h-1} v + D_2 \hat{f}_h w.
\end{equation}
Under assumptions (\ref{eq44}) and (\ref{eq45}) we have 
$Im \Lambda = Im D\hat{f}_h$ and $Im \Lambda$ is closed in $X$. After (\ref{eq414}), using Proposition \ref{prop31}  on $\Lambda$ we obtain
\begin{equation}\label{eq416}
\left.
\begin{array}{l}
\exists c \in (0, + \infty), \exists (v_{h-1}^*, v_h^*) \in U \times U, \;  {\rm s.t.}\\
\Lambda(v_{h-1}^*, v_h^*) = z_h + D_1 \hat{f}_h z_{h-1} - D_1 \hat{f}_h \circ D_1 \hat{f}_{h-1} y_{h-1}^*\\
{\rm i.e.}\\
D_1 \hat{f}_h\circ D_2 \hat{f}_{h-1} v_{h-1}^* + D_2 \hat{f}_h v_h^* =\\
z_h + D_1 \hat{f}_h z_{h-1} - D_1 \hat{f}_h \circ D_1 \hat{f}_{h-1} y_{h-1}^*
\;\; {\rm and} \\
\max \{ \Vert  v_{h-1}^* \Vert, \Vert  v_{h}^* \Vert \} \leq 
c \cdot \Vert z_h + D_1 \hat{f}_h z_{h-1} - D_1 \hat{f}_h \circ D_1 \hat{f}_{h-1} y_{h-1}^* \Vert.
\end{array}
\right\}
\end{equation}
From this last inequality, using (\ref{eq413}), we obtain
\[
\begin{array}{l}
\max \{ \Vert  v_{h-1}^* \Vert, \Vert  v_{h}^* \Vert \}\\
\leq  c \cdot (\Vert z_h \Vert + \Vert D_1 \hat{f}_h \Vert \cdot \Vert z_{h-1} \Vert + \Vert D_1 \hat{f}_h \circ D_1 \hat{f}_{h-1} \Vert \cdot \Vert y_{h-1}^* \Vert)\\
\leq c \cdot (\Vert z_h \Vert + \Vert D_1 \hat{f}_h \Vert \cdot \Vert z_{h-1} \Vert + \Vert D_1 \hat{f}_h \circ D_1 \hat{f}_{Th1} \Vert \cdot a_{h-2} \cdot \max_{1 \leq t \leq h-2} \Vert z_t \Vert )\\
\leq c \cdot (1 + \Vert D_1 \hat{f}_h \Vert + a_{h-2} \cdot  \Vert D_1 \hat{f}_h \circ D_1 \hat{f}_{h-1} \Vert) \cdot \max_{1 \leq t \leq h} \Vert z_t \Vert.
\end{array}
\] 
We set $c_1 := c \cdot (1 + \Vert D_1 \hat{f}_h \Vert + a_{h-2} \cdot  \Vert D_1 \hat{f}_h \circ D_1 \hat{f}_{h-1} \Vert) \in (0, + \infty)$.
Then we have proven the following assertion.
\begin{equation}\label{eq417}
\exists c_1 \in (0, + \infty), \max \{ \Vert  v_{h-1}^* \Vert, \Vert  v_{h}^* \Vert \} \leq c_1 \cdot \max_{1 \leq t \leq h} \Vert z_t \Vert. 
\end{equation}
We set 
\begin{equation}\label{eq418}
y_h^* := D_2 \hat{f}_{h-1} v_{h-1}^* + D_1 \hat{f}_{h-1} y_{h-1}^* - z_{h-1}.
\end{equation}
This equality implies
\begin{equation}\label{eq419}
- y_h^* + D_1 \hat{f}_{h-1} y_{h-1}^* +  D_2 \hat{f}_{h-1} v_{h-1}^* = z_{h-1}
\end{equation}
which is the penultimate wanted equation.
\vskip1mm
\noindent
Notice that we have $\Vert y_h^* \Vert \leq \Vert D_2 \hat{f}_{h-1} \Vert \cdot \Vert v_{h-1}^* \Vert + \Vert D_1 \hat{f}_{h-1} \Vert \cdot \Vert y_{h-1}^* \Vert + \Vert z_{h-1} \Vert,$
and using (\ref{eq417}) and (\ref{eq418}) we obtain
\[
\begin{array}{ccl}
\Vert y_h^* \Vert & \leq & \Vert D_2 \hat{f}_{h-1} \Vert \cdot c_1 \cdot \max_{1 \leq t \leq h} \Vert z_t \Vert\\
\null & \null & + \Vert D_1 \hat{f}_{h-1} \Vert \cdot a_{h-2} \cdot \max_{1 \leq t \leq h-2} \Vert z_t \Vert + \Vert z_{h-1} \Vert \\
\null & \leq & (c_1 \cdot \Vert D_2 \hat{f}_{h-1} \Vert + a_{h-2} \cdot \Vert D_1 \hat{f}_{h-1} \Vert + 1) \cdot \max_{1 \leq t \leq h} \Vert z_t \Vert.
\end{array}
\]
We set $c_2 := c_1 \cdot \Vert D_2 \hat{f}_{h-1} \Vert + a_{h-2} \cdot \Vert D_1 \hat{f}_{h-1} \Vert + 1$, and so we have proven
\begin{equation}\label{eq420}
\exists c_2 \in (0, + \infty), \Vert y_h^* \Vert \leq c_2 \cdot \max_{1 \leq t \leq h} \Vert z_t \Vert.
\end{equation}
We set $a_h := \max \{ a_{h-3}, c_1, c_2 \}$, and from (\ref{eq413}), (\ref{eq417}) and (\ref{eq420}) we have proven
\begin{equation}\label{eq421}
\exists a_h \in (0, + \infty), \max \{ \max_{1 \leq t \leq h} \Vert y_t^* \Vert, \max_{0 \leq t \leq h} \Vert v_t^* \Vert \} \leq a_h \cdot \max_{1 \leq t \leq h} \Vert z_t \Vert.
\end{equation}
Now we show that the last equation is satisfied by $y_h^*$ and $v_h^*$. Using (\ref{eq418}) and (\ref{eq416}), we obtain
\[
\begin{array}{l}
D_1 \hat{f}h y_h^* + D_2 \hat{f}_h v_h^*\\
=  D_1 \hat{f}_T(D_2 \hat{f}_{h-1} v_{h-1}^* + D_1 \hat{f}_{h-1} y_{h-1}^* - z_{h-1}) +  D_2 \hat{f}_h v_h^*\\
=  ( D_1 \hat{f}_h \circ (D_2 \hat{f}_{h-1} v_{h-1}^* + D_2 \hat{f}_h v_h^*)  + D_1 \hat{f}_h \circ D_1 \hat{f}_{h-1} y_{h-1}^* - D_1 \hat{f}_h  z_{h-1}\\
= (z_h + D_1 \hat{f}_hz_{h-1} - D_1 \hat{f}_h \circ D_1 \hat{f}_{h-1} y_{h-1}^*) 
 + D_1 \hat{f}_h \circ D_1 \hat{f}_{h-1} y_{h-1}^* - D_1 \hat{f}_h  z_{h-1}\\
=  z_h.
\end{array}
\]
We have proven that
\begin{equation}\label{eq422}
D_1 \hat{f}_h y_h^* + D_2 \hat{f}_h v_h^* = z_h.
\end{equation}
From (\ref{eq413}), (\ref{eq419}), (\ref{eq421}) and (\ref{eq422}) we have proven the following assertion
\[
\left\{
\begin{array}{l}
\exists a_h \in (0, + \infty), \forall (z_t)_{0 \leq t \leq h} \in Im Dg_h({\bf \hat{x}}^h, {\bf \hat{u}}^h ),\\
\exists (y_t^*)_{1 \leq t \leq h} \in X^h, \exists (v_t^*)_{0 \leq t \leq h} \in U^{h+1}, \; {\rm s.t.}\\
- y_1^* + D_2 \hat{f}_0 v_0^* = z_0, \;
\forall t \in \{1,...,h-1 \},\;\;  -y_{t+1} + D_1 \hat{f}_t y_t^* + D_2 \hat{f}_t v_t^* = z_t,\\
D_1 \hat{f}_h y_h^* + D_2 \hat{f}_h v_h^* = z_h, \; \; {\rm and}\\
 \max \{ \max_{1 \leq t \leq h} \Vert y_t^* \Vert, \max_{0 \leq t \leq h} \Vert v_t^* \Vert \} \leq a_h \cdot \max_{1 \leq t \leq h} \Vert z_t \Vert.
\end{array}
\right.
\]
This last assertion is equivalent to the following one
\[
\left\{
\begin{array}{l}
\exists a_h \in (0, + \infty), \forall {\bf z}^h = (z_t)_{0 \leq t \leq h} \in ImDg^h({\bf \hat{x}}^h, {\bf \hat{u}}^h ),\\
\exists {\bf y}^{h,*} = (y_t^*)_{1 \leq t \leq h} \in X^h, \exists {\bf v}^{h,*} = (v_t^*)_{0 \leq t \leq h} \in U^{h+1}, \; {\rm s.t.}\\
Dg^h({\bf \hat{x}}^h, {\bf \hat{u}}^h )( {\bf y}^h {\bf v}^{h,*}) = {\bf z}^h \;\; {\rm and} \;\;
\Vert ({\bf y}^{h,*}, {\bf v}^{h,*}) \Vert \leq a_h \cdot \Vert {\bf z}^h \Vert.
\end{array}
\right.
\]
Now using Proposition \ref{prop31} on the operator $Dg^h({\bf \hat{x}}^h, {\bf \hat{u}}^h )$, the previous assertion permits us to assert that 
$Im Dg^h({\bf \hat{x}}^h, {\bf \hat{u}}^h )$ is closed in $X^{h+1}$, and the proof of (i) is complete.
\vskip2mm
\noindent
{\bf (ii)} We arbitrarily fix ${\bf z}^h = (z_1,...,z_h) \in X^h$.  Since $D \hat{f}_h$ is surjective, there exists $y_h^{\#} \in X$ and $v_h^{\#} \in U$ such that $D \hat{f}_h(y_h^{\#}, v_h^{\#}) = z_h$. Since $D \hat{f}_{h-1}$ is surjective, there exists $y_{h-1}^{\#} \in X$ and $v_{h-1}^{\#} \in U$ such that $D \hat{f}_{h-1}(y_{h-1}^{\#}, v_{h-1}^{\#}) = z_{h-1} + y_h^{\#}$. We iterate this bachward reasoning until $t=2$ to obtain
\begin{equation}\label{eq423}
\left.
\begin{array}{l}
\forall t \in \{2,...,h \}, \exists (y_t^{\#}, v_t^{\#}) \in X \times U \;\; {\rm s.t.} \;\;  D \hat{f}_h(y_h^{\#},v_h^{\#}) = z_h\\
{\rm and} \;\; \forall t \in {2,...,h-1}, -y_t^{\#} + D \hat{f}_t(y_t^{\#}, v_t^{\#}) = z_t.
\end{array}
\right\}
\end{equation}
Now we introduce the linear continuous operator $M \in {\mathfrak L}(U \times U, X)$ by setting $M(v_0,, v_1) ;= D_1 \hat{f}_1 \circ D_2 \hat{f}_0 v_0 + D_2 \hat{f}_1 v_1$. From (\ref{eq46}) we have $Im M = X$ i.e. $M$ is surjective. Therefore we obtain
\begin{equation}\label{eq424}
\exists (v_0^{\#}, v_1^{\#}) \in U \times U \;\; {\rm s.t.} \;\; D_1 \hat{f}_1 \circ D_2 \hat{f}_0 v_0^{\#} + D_2 \hat{f}_1 v_1^{\#} = 
z_1 + y_2^{\#} + D_1 \hat{f}_1 z_0.
\end{equation}
We set $y_1^{\#} := D_2 \hat{f}_0 v_0^{\#} - z_0$. Hence we obtain
\begin{equation}\label{eq425}
- y_1^{\#} + D_2 \hat{f}_0 v_0^{\#} = z_0.
\end{equation}
Using (\ref{eq424}) and (\ref{eq425}), we calculate
\[
\begin{array}{l}
- y_2^{\#} + D_1 \hat{f}_1 y_1^{\#}  + D_2 \hat{f}_1 v_1^{\#} = - y_2^{\#} + D_1 \hat{f}_1 (D_2 \hat{f}_0 v_0^{\#} - z_0)  + D_2 \hat{f}_1 v_1^{\#}\\
=  - y_2^{\#}+ (D_1 \hat{f}_1 \circ D_2 \hat{f}_0 v_0^{\#} + D_2 \hat{f}_1 v_1^{\#}) - D_1 \hat{f}_1 z_0\\
= - y_2^{\#}+ (z_1  + y_2^{\#} + D_1 \hat{f}_1 z_0) -  D_1 \hat{f}_1 z_0 = z_1.
\end{array}
\]
We have proven
\begin{equation}\label{eq426}
- y_2^{\#} + D_1 \hat{f}_1 y_1^{\#}  + D_2 \hat{f}_1 v_1^{\#} = z_1.
\end{equation}
From (\ref{eq423}), (\ref{eq425}) and (\ref{eq426}) we have proven 
\begin{equation}\label{eq427}
\left.
\begin{array}{l}
\forall (z_t)_{0 \leq t \leq h} \in X^{h+1}, \exists (y_t^{\#})_{1 \leq t \leq h} \in X^h, \exists (v_t^{\#})_{0 \leq t \leq h}\in U^{h+1} \;\; {\rm s.t.}\\
-y_1^{\#} + D_2 \hat{f}_0 v_0^{\#} = z_0, \forall t \in \{ 1, ..., h-1 \} \;\; -y_{t+1}^{\#} + D \hat{f} 
(y_t^{\#}, v_t^{\#}) = z_t\\
{\rm and} \;\; D \hat{f}_h (y_h^{\#}, v_h^{\#}) = z_h.
\end{array}
\right\}
\end{equation}
This assertion is equivalent to
$$\forall {\bf z}^h \in X^{h+1}, \exists {\bf y}^{h, \#} \in X^h, \exists {\bf v}^{h, \#} \in U^{h+1} \; {\rm s.t.} \; Dg^h({\bf x}^h, {\bf u}^h() {\bf y}^{h, \#},{\bf v}^{h, \#}) = {\bf z}^h$$
which means that $Dg^h({\bf x}^h, {\bf u}^h)$ is surjective. 
\end{proof}
\vskip2mm
\begin{lemma}\label{lem43}
Let $(\hat{x}_t)_{t \in \N}, (\hat{u}_t)_{t \in \N})$ be an optimal solution of one of the problems $({\bf P}_i(\sigma))$, $i \in \{1,2,3 \}$. 
Under (H1), (H2), (H3), (\ref{eq43}) and (\ref{eq44}), for all $h \in \N_*$, there exists $\lambda^h_0 \in \R$ and $(p_{t+1}^h)_{0 \leq t \leq h} \in (X^*)^{h+1}$ such that the following assertions hold.
\begin{enumerate}
\item[(a)] $\lambda^h_0$ and $(p_{t+1}^h)_{0 \leq t \leq h}$ are not simultaneously equal to zero.
\item[(b)] $\lambda_0^h \geq 0$.
\item[(c)] $p_t^h = p_{t+1}^h \circ D_1 f_t(\hat{x}_t, \hat{u}_t) + \lambda_0^h D_1 \phi_t(\hat{x}_t, \hat{u}_t)$ for all $t \in \N_*$.
\item[(d)] $\langle \lambda_0^h D_2\phi_t(\hat{x}_t, \hat{u}_t) + p_{t+1}^h \circ D_2 f_t(\hat{x}_t, \hat{u}_t), u_t - \hat{u}_t \rangle \leq 0$ \;\; for all  $ t \in \{0,...,h \}$, for all $u_t \in U_t$.
\end{enumerate}
Moreover, for all $h \geq 2$, if in addition we assume (H4), (H5) and (H6) fulfilled, the following assertions hold.
\begin{enumerate}
\item[(e)] For all $t \in \{1,...,h+1 \}$, there exists $a_t, b_t \in \R_+$ such that, for all $s \in \{1,...,h \}$, $\Vert p_t^h \Vert \leq a_t \lambda_0^h + b_t \Vert p_s^h \Vert$.
\item[(f)] For all $ t \in \{1,...,h \}$, $(\lambda_0^h, p_t^h) \neq (0,0)$.
\item[(g)] For all $ t \in \{1,...,h \}$, for all $z \in A_t := D_2 f_{t-1}(\hat{x}_{t-1}, \hat{u}_{t-1})(T_{U_{t-1}}(\hat{u}_{t-1}))$, there exists $c_z \in \R$ such that $p_t^h(z) \leq c_z \lambda_0^h$ for all $h \geq t$.
\end{enumerate}
\end{lemma}
\begin{proof}
Let $h \in \N_*$. Using Lemma \ref{lem41}, (\ref{eq41}) and (\ref{eq42}), we know that $({\bf \hat{x}}^h, {\bf \hat{u}}^h)$ (where ${\bf \hat{x}}^h = (x_1^h,...,x_h^h)$ and ${\bf \hat{u}}^h = (u_0^h,...,u_h^h)$), is an optimal solution of the following maximization problem,  
\[
\left\{
\begin{array}{cl}
{\rm Maximize} & J_h({\bf x}^h, {\bf u}^h)\\
{\rm when} & ({\bf x}^h, {\bf u}^h) \in (\prod_{t=1}^h X_t) \times (\prod_{t=0}^h U_t),\\
\null & g^h({\bf x}^h, {\bf u}^h) = 0.
\end{array}
\right.
\]
From (H3) we know that $J_h$ is Fr\'echet differentiable at $({\bf \hat{x}}^h, {\bf \hat{u}}^h)$ and $g^h$ is Fr\'echet continuously differentiable at $({\bf \hat{x}}^h, {\bf \hat{u}}^h)$. From (\ref{eq43}), (\ref{eq44}) and Lemma \ref{lem42} we know that $Im Dg^h({\bf \hat{x}}^h, {\bf \hat{u}}^h)$ is closed in $X^{h+1}$. Now using the multiplier rule which is given in \cite{Ja} (Theorem 3.5 p. 106--111 and Theorem 5.6 p. 118) and explicitely written in \cite{BB2} (Theorem 4.4), and proceeding as in the proof of Lemma 4.5 of \cite{BB2}, we obtain the assertions (a), (b), (c), (d).
\vskip1mm
The proof of assertions (e), (f), (g) is given by Lemma 4.7 of \cite{BB2}. The proof of this Lemma 4.7 uses the condition $0 \in Int[Df(\hat{x}_t, \hat{u}_t)(X \times T_{U_t}(\hat{u}_t)) \cap B_{X \times U}]$ where $B_{X \times U}$ is the closed unit ball of $X \times U$. It suffices to notice that our assumption (H4) implies this condition.
\end{proof}
\begin{remark}\label{rem44}
In Lemma 4.5 of \cite{BB2} the finiteness of the codimension of $ImD_2f(\hat{x}_t, \hat{u}_t)$ is useful to ensure the closedness of $ImDg^h({\bf \hat{x}}^h, {\bf \hat{u}}^h)$. Here we can avoid this assumption of finiteness thanks the recursive assumptions.
\end{remark}
The following proposition is used in the proof of the main result.
\begin{proposition}\label{prop45}
Let $(\hat{x}_t)_{t \in \N}, (\hat{u}_t)_{t \in \N})$ be an optimal solution of one of the problems $({\bf P}_i(\sigma))$, $i \in \{1,2,3 \}$. 
Under (H1-H6) we introduce
$$Z_0 := D_2f_0(\sigma, \hat{u}_0)(T_{U_0}(\hat{u}_0)) \;\;\; {\rm and} \;\;\; Z_1 := D_2f_1(\hat{x}_1, \hat{u}_1)(T_{U_1}(\hat{u}_1)).$$
Then, for all $h \in \N$, $h \geq 2$, there exist $\lambda_0^h \in \R$ and  $(p_{t+1}^h)_{0 \leq t \leq h} \in (X^*)^{h+1}$ such that the following assertions hold.
\begin{enumerate}
\item[(1)] $\lambda^h_0 \geq 0$.
\item[(2)] $p_t^h = p_{t+1}^h \circ D_1 f_t(\hat{x}_t, \hat{u}_t) + \lambda_0^h D_1 \phi_t(\hat{x}_t, \hat{u}_t)$ for all $t \in \N_*$.
\item[(3)] $\langle \lambda_0^h D_2\phi_t(\hat{x}_t, \hat{u}_t) + p_{t+1}^h \circ D_2 f_t(\hat{x}_t, \hat{u}_t), u_t - \hat{u}_t \rangle \leq 0$ \;\; for all  $ t \in \{0,...,h \}$, for all $u_t \in U_t$.
\item[(4)] For all $t \in \{1,...,h+1 \}$, there exists $a_t, b_t \in \R_+$ such that, for all $s \in \{1,...,h \}$, $\Vert p_t^h \Vert \leq a_t \lambda_0^h + b_t \Vert p_s^h \Vert$.
\item[(5)]  $(\lambda_0^h, {p_1^h}_{\mid Z_0}, {p_2^h}_{\mid Z_1}) \neq (0,0,0)$.
\item[(6)] For all $z_0 \in Z_0$, for all $z_1 \in Z_1$, there exists $c_{z_0,z_1} \in \R$ such that, \\
for all $h \geq 2$, $p_1^h(z_0) + p_2^h(z_1) \leq c_{z_0,z_1} \lambda_0^h.$
\item[(7)] For all $v \in X$ there exists $(z_0,z_1) \in Z_0 \times Z_1$ such that\\
$p_2^h(v) = p_1^h(z_0) + p_2^h(z_1) - \lambda_0^h D_1 \phi_1(\hat{x}_1, \hat{u}_1)(z_0)$ for all $h \geq 2$.
\end{enumerate}
\end{proposition}
\begin{proof}
\underline{Proof of (1-4)} Note that conditions (\ref{eq43}) and (\ref{eq44}) are consequences of (H4). We use $\lambda_0^h$ and  $(p_{t+1}^h)_{0 \leq t \leq h}$ which are provided by Lemma \ref{lem43}. Hence conclusions (1), (2) and (3) are given by Lemma \ref{lem43}. The conclusion (4) is the conclusion (e) of Lemma \ref{lem43}.\\
\underline{Proof of (5)} From the conclusion (f) of Lemma \ref{lem43}, we know that $(\lambda_0^h, p_1^h) \neq (0,0)$. We want to prove that [$(\lambda_0^h, {p_1^h)}\neq (0,0)]$ implies (5). To do that we proceed by contraposition; we assume that $[\lambda_0^h = 0, {p_1^h}_{\mid Z_0} = 0, {p_2^h}_{\mid Z_1} = 0]$ and we want to prove that $[\lambda_0^h = 0 , p_1^h = 0]$. Since $\lambda_0^h = 0$, using the conclusion (2) we obtain $p_1^h = p_2^h \circ D_1f_1(\hat {x}_1, \hat{u}_1)$ which implies
$$p_1^h \circ D_2f_0(\sigma, \hat{u}_0)(T_{U_0}(\hat{u}_0)) = p_2^h \circ D_1f_1(\hat {x}_1, \hat{u}_1) \circ D_2f_0(\sigma, \hat{u}_0)(T_{U_0}(\hat{u}_0)),$$
and since ${p_1^h}_{\mid Z_0} = 0$, we obtain $p_2^h \circ D_1f_1(\hat {x}_1, \hat{u}_1) \circ D_2f_0(\sigma, \hat{u}_0)(T_{U_0}(\hat{u}_0)) = 0$, and since ${p_2^h}_{\mid Z_1} = 0$, using (H5), we obtain $p_2^h = 0$ (on $X$ all over). Hence $p_1^h = p_2^h \circ D_1f_1(\hat {x}_1, \hat{u}_1) = 0$. The proof of (5) is complete.\\
\underline{Proof of (6)} Let $z_0 \in Z_0$ , $z_1 \in Z_1$. Using conclusion (g) of Lemma \ref{lem43}, we obtain that there exists $c_{z_0}^0 \in \R$ such that $p_1^h(z_0) \leq c^0_{z_0} \lambda^h_0$ for all $h \geq 1$, and that there exists  $c_{z_1}^1 \in \R$ such that $p_2^h(z_1) \leq c^1_{z_1} \lambda^h_0$  for all $h \geq 2$. Setting $c_{z_0,z_1} := c_{z_0}^0 + c_{z_1}^1$ we obtain the announced conclusion.\\
\underline{Proof of (7)} From (H5), for all $v \in X$, there exists $\zeta_0 \in T_{U_0}(\hat{u}_0)$ and $\zeta_1 \in T_{U_1}(\hat{u}_1)$ such that 
$$v = D_1f_1(\hat{x}_1, \hat{u}_1) \circ D_2f_0(\sigma, \hat{u}_0)(\zeta_0) + D_2 f_1(\hat{x}_1, \hat{u}_1)(\zeta_1).$$
We set $z_0 := D_2f_0(\sigma, \hat{u}_0)(\zeta_0) \in Z_0$ and $z_1 := D_2f_1(\hat{x}_1, \hat{u}_1)(\zeta_1) \in Z_1$, hence we have 
\begin{equation}\label{eq428}
v =  D_1f_1(\hat{x}_1, \hat{u}_1)(z_0) + z_1.
\end{equation}
From conclusion (2) we deduce
$$p_1^h \circ D_2f_0(\sigma, \hat{u}_0) = p_2^h \circ D_1f_1(\hat{x}_1, \hat{u}_1) \circ D_2f_0(\sigma, \hat{u}_0) + \lambda_0^h D_1 \phi_1(\hat{x}_1, \hat{u}_1) \circ D_2f_0(\sigma, \hat{u}_0).$$
Applying this last equation to $\zeta_0$ we obtain
$$p_1^h(z_0) = p_2^h \circ D_1f_1(\hat{x}_1, \hat{u}_1)(z_0)   + \lambda_0^h D_1 \phi_1(\hat{x}_1, \hat{u}_1)(z_0).$$
Adding $p_2^h(z_1)$ to this equality we obtain
$$p_1^h(z_0) +p_2^h(z_1)  = p_2^h \circ D_1f_1(\hat{x}_1, \hat{u}_1)(z_0) +  p_2^h(z_1)  + \lambda_0^h D_1 \phi_1(\hat{x}_1, \hat{u}_1)(z_0).$$ 
Using (\ref{eq428}) we have $p_1^h(z_0) +p_2^h(z_1)  = p_2^h(v) +  \lambda_0^h D_1 \phi_1(\hat{x}_1, \hat{u}_1)(z_0)$ which implies the announced equality.
\end{proof}
\section{Proof of the main theorem}
Proposition \ref{prop45} provides sequences $(\lambda_0^h)_{h \geq 2}$, $(p_t^h)_{h \geq t}$ for all $t \in \N_*$. We set $q_1^h := p_1^h \circ D_2f_0(\sigma, \hat{u}_0) \in U^*$ and $
q_2^h := p_2^h \circ D_2f_1(\hat{x}_1, \hat{u}_1) \in U^*$ for all $h \geq 2$. From conclusion (5) of Proposition \ref{prop45} we obtain
$$(\lambda_0^h, {q_1^h}_{\mid T_{U_0}(\hat{u}_0)}, {q_2^h}_{\mid T_{U_1}(\hat{u}_1)}) \neq (0,0,0).$$
We introduce $\Sigma := \overline{{\rm aff}}(T_{U_0}(\hat{u}_0) \times T_{U_1}(\hat{u}_1))$ the closed affine hull of $T_{U_0}(\hat{u}_0) \times T_{U_1}(\hat{u}_1)$ which is a closed vector subspace since the tangent cones contain the origine. From the previous relation we can assert that $(\lambda_0^h, (q_1^h,q_2^h)_{\mid \Sigma}) \neq (0, (0,0))$. We introduce the number
$$\theta^h := \lambda_0^h + \Vert(q_1^h,q_2^h)_{\mid \Sigma} \Vert_{\Sigma^*} > 0.$$
Since the list of the multipliers of the problem in finite horizon is a cone, we can replace $\lambda_0^h$ by $\frac{1}{\theta^h} \lambda_0^h$ and the $p_t^h$ by $\frac{1}{\theta^h} p_t^h$ (without to change the writting), and so we can assume that the following property holds.
\begin{equation}\label{eq51}
\forall h \geq 2, \;\; \lambda_0^h + \Vert(q_1^h,q_2^h)_{\mid \Sigma} \Vert_{\Sigma^*}  = 1.
\end{equation}
Using the Banach-Alaoglu theorem, we can assert that there exist an increasing mapping $\varphi_1 : [2, + \infty)_ {\N} \rightarrow [2, + \infty)_{\N}$, $\lambda_0 \in \R$, $(q_1, q_2) \in \Sigma^*$ such
$$(\lambda_0^{\varphi_1(h)}, (q_1^{\varphi_1(h)}, q_2^{\varphi_1(h)})_{\mid \Sigma}) \overset{w^*}{\longrightarrow} (\lambda_0, (q_1,q_2)) \;\;{\rm when} \; \; h \rightarrow + \infty.$$
Now we want to establish that
\begin{equation}\label{eq52}
(\lambda_0, (q_1,q_2)) \neq (0,(0,0)).
\end{equation}
To do that we proceed by contradiction; we assume that $\lambda_0 = 0$ and $(q_1,q_2) = (0,0)$. 
From conclusion (6) of Proposition \ref{prop45} we deduce that, for all $\zeta_0 \in T_{U_0}(\hat{u}_0)$ and for all $\zeta_1 \in T_{U_1}(\hat{u}_1)$, there exixts $c_{\zeta_0, \zeta_1} \in \R$ such that $q_1^{\varphi_1(h)}(\zeta_0) + q_2^{\varphi_1(h)}(\zeta_1) \leq c_{\zeta_0, \zeta_1} \lambda_0^{\varphi_1(h)}$ for all $h \geq 2$. Hence we can use Proposition \ref{prop32} with $Y = \Sigma$, $K = T_{U_0}(\hat{u}_0) \times T_{U_1}(\hat{u}_1)$, $S = \Sigma$, $\rho_h = \lambda_0^{\varphi_1(h)}$, and $\pi_h = (q_1^{\varphi_1(h)}, q_2^{\varphi_1(h)})_{\mid \Sigma}$. Consequently we  obtain that $\lim_{h \rightarrow + \infty} \Vert (q_1^{\varphi_1(h)}, q_2^{\varphi_1(h)})_{\mid \Sigma} \Vert_{\Sigma ^*} = 0$. Since we also have $\lim_{h \rightarrow + \infty}  \lambda_0^{\varphi_1(h)} =0$, we obtain a contradiction with (\ref{eq51}). Hence (\ref{eq52}) is proven.
\vskip2mm
\noindent
From conclusion (7) of Proposition \ref{prop45} we have, for all $x \in X$, there exists $(\zeta_0, \zeta_1) \in \Sigma$ such that, for all $h \geq 2$, 
$$p_2^{\varphi_1(h)} (x) = (q_1^{\varphi_1(h)},q_2^{\varphi_1(h)})_{\mid \Sigma}(\zeta_0,\zeta_1) - \lambda_0^{\varphi_1(h)}D_1 \phi_1(\hat{x}_1, \hat{u}_1) \circ D_2f_0(\sigma, \hat{u}_0)(\zeta_0)$$
which permits to say that there exists $p_2 \in X^*$ such that $p_2^{\varphi_1(h)} \overset{w^*}{\rightarrow} p_2$ when $h \rightarrow + \infty$.
\vskip2mm
\noindent
From conclusion (2) of Proposition \ref{prop45} at $t=1$, we obtain that there exists $p_1 \in X^*$ such that $p_1^{\varphi_1(h)} \overset{w^*}{\rightarrow} p_1$ when $h \rightarrow + \infty$, and from (\ref{eq52}) we obtain 
\begin{equation}\label{eq53}
(\lambda_0, (p_1,p_2)) \neq (0, (0,0)).
\end{equation}
Since $(p_1^{\varphi_1(h)})_{h \geq 2}$ is weak-star convergent on $X$, using the Banach-Steinhaus theorem we can assert that the sequence $( \Vert p_1^{\varphi_1(h)} \Vert_{X^*})_{h \geq 2}$ is bounded. Since $(\lambda_0^{\varphi_1(h)})_{h \geq 2}$ is convergent in $\R$, it is bounded, and from conclusion (4) of Proposition \ref{prop45}, we deduce that the sequence $(p_t^{\varphi_1(h)})_{h \in [t, + \infty)_{\N}}$ is bounded for each $t \in \N_*$. Then we can use Proposition \ref{prop33} to ensure the existence of an increasing function $\varphi_2 : [2, + \infty)_{\N} \rightarrow [2, + \infty)_{\N}$ and a sequence $(p_t)_{t \in N_*} \in (X^*)^{\N_*}$ such that $p_t^{\varphi_1 \circ \varphi_2(h)} \overset{w^*}{\rightarrow} p_t$ when $h \rightarrow + \infty$ for all $t \in \N_*$, and we have also $\lambda_0^{\varphi_1 \circ \varphi_2(h)} \rightarrow \lambda_0$ when $h \rightarrow + \infty$. Hence we have built all the multipliers. The properties of these multipliers are obtained by taking limits from the properties of the $\lambda_0^h$ and the $p_t^h$. Their non triviality is proven by (\ref{eq53}). Hence the proof of the main result is complete.
\vskip2mm
\noindent

\bibliographystyle{amsplain}

\end{document}